\documentclass[12pt,reqno]{amsart}

\usepackage{stmaryrd,eucal,color,graphicx,subfigure,mathrsfs}

\numberwithin{equation}{section}

\newtheorem{theorem}{Theorem}[section]
\newtheorem{lemma}[theorem]{Lemma}
\newtheorem{corollary}[theorem]{Corollary}
\newtheorem{proposition}[theorem]{Proposition}

\theoremstyle{definition}

\theoremstyle{remark}
\newtheorem{remark}[theorem]{Remark}

\numberwithin{equation}{section}

\newcommand{\Bv}{{\boldsymbol{v}}}

\newcommand{\Bx}{{\boldsymbol{x}}}
\newcommand{\By}{{\boldsymbol{y}}}

\newcommand{\BK}{{\boldsymbol{K}}}

\newcommand{\CC}{{\mathcal C}}

\begin{document}

\title{Global $W^{2,p}$ estimates for elliptic equations in the non-divergence form}

\author{Weifeng Qiu}
\address{Department of Mathematics, City University of Hong Kong,
83 Tat Chee Avenue, Kowloon, Hong Kong, China}
\email{weifeqiu@cityu.edu.hk}

\author{Lan Tang}
\address{School of Mathematics and Statistics, Cental China Normal University,
Wuhan, Hubei 430079, China}
\email{lantang@mail.ccnu.edu.cn}

\begin{abstract}
This paper is devoted to establishing global $W^{2, p}$ estimate for strong solutions to the Dirichlet problem of  uniformly elliptic equations in the non-divergence form where the domain is a Lipschitz polyhedra.
\end{abstract}


\thanks{Weifeng Qiu is supported by a grant from the Research Grants Council of the Hong Kong Special Administrative Region, China (Project No. CityU 11302219). Lan Tang is  supported by the National Natural Science Foundation of China (No. 11831009).}


\maketitle

\section{Introduction}
Let $\Omega$ be a bounded domain  in $\mathbb{R}^3$ with Lipschitz boundary and  $f\in L^p(\Omega)$ with some $1<p<\infty$. Assume that
\begin{eqnarray}
\label{assump_1}
A \in [C^{0}(\overline{\Omega})]^{3 \times 3}; \ \text{and}\  \lambda_0 I_3 \leq A(\Bx) \leq \lambda_1 I_3, \ \forall \Bx\in \Omega
\end{eqnarray}for two uniform positive constants $0<\lambda_0\leq \lambda_1<\infty$. Here $I_3$ denotes the $3\times3$ identity matrix.

In this work, we mainly consider the Dirichlet problem for the elliptic equations with non-divergence structure:
\begin{equation}
\label{eq1}
\left \{
 \begin{array}{ll}
A: D^2 u=f \ \ & \text{in}\ \Omega\\ \\
u=0 \  &\text{on}\ \partial\Omega
 \end{array}
\right.
 \end{equation}Here for any $\Bx\in\Omega$,  $A(\Bx): D^2 u(\Bx)$ denotes the trace of the matrix $ A(\Bx) D^2 u(\Bx)$.

For global $W^{2, p}$ estimate for strong solutions to (\ref{eq1}), it is one of key problems in the regularity theory for elliptic equations and has attracted much attention from mathematicians since 1950s. It is well known that when $\Omega\in C^{1,1}$,  then for any $1<p<\infty$, (\ref{eq1}) admits a unique strong solution $u\in W^{2, p}\cap W^{1,p}_0$( see chapter 9 of Gilbarg-Trudinger \cite{GT01}) with the estimate
\begin{eqnarray}
\label{estimate1}
\|u\|_{2, p}\leq C \|f\|_p
\end{eqnarray}

As for the boundary with less regularity, there are few results for global regularity. Recently, Li \cite{Li}   obtained $W^{2, p}$ estimate for strong solutions with $C^{1, \alpha}$ boundary for $p$ depending on $\alpha$.

In our work, we get global $W^{2, p}$ estimate  for strong solutions to (\ref{eq1}) with $p$ in possible ranges under weak regularity assumptions for $\partial \Omega$. More precisely,  we have
\begin{theorem}
\label{main_theorem1}
Let $\Omega$ be a Lipschitz polyhedral domain in $\mathbb{R}^{3}$ and  $A$ satisfy $(\ref{assump_1})$. Then the following conclusions hold:

${(1)}$ For any $p\in [\frac{6}{5}, \frac{4}{3})$,  there exists a positive constant $C_{p}$ such that
\begin{align}
\label{main_estimate1}
\Vert u\Vert_{W^{2,p}(\Omega)} \leq C_{p} \left(\Vert A : D^{2}u\Vert_{L^{p}(\Omega)}+ \Vert u\Vert_{L^{p}(\Omega)} \right),
\qquad \forall u \in W^{2,p}(\Omega) \cap W_{0}^{1,p}(\Omega).
\end{align}

${(2)}$ Assume $\Omega$ is also convex, then $(\ref{main_estimate1})$  holds true for any $ p$  with $\frac{6}{5}\leq p \leq  2$.
\end{theorem}\

In fact, when $\Omega$ be a convex polyhedral domain in $\mathbb{R}^{3}$, we can obtain a much stronger estimate:
\begin{corollary}
\label{cor1}
Let $\Omega$ be a convex polyhedral domain in $\mathbb{R}^{3}$ and  $A$ satisfy $(\ref{assump_1})$. Then For any $p\in (\frac{3}{2}, 2]$,  there exists a positive constant $C_{p}$ such that
\begin{align}
\label{main_estimate2}
\Vert u\Vert_{W^{2,p}(\Omega)} \leq C_{p}  \Vert A : D^{2}u\Vert_{L^{p}(\Omega)},  \qquad \forall u \in W^{2,p}(\Omega) \cap W_{0}^{1,p}(\Omega).
\end{align}
\end{corollary}\

The following part is organized as follows:  Section 2 is devoted to giving some known results; in Section 3, we would like to establish the proofs of Theorem \ref{main_theorem1} and Corollary \ref{cor1}.

\vspace{30pt}

\section{preliminary results}
In this part we state some known results, which are important to prove Theorem \ref{main_theorem1} and Corollary \ref{cor1}. As for the proofs for these known results, we would like to refer the interested readers to \cite{Dauge1992} and \cite{Stein1970}.

The first one is global $W^{2, p}$ estimate for Poisson equation given by Dauge\cite{Dauge1992}:

\begin{theorem}(Corollary 3.10, 3.12 of \cite{Dauge1992})
\label{dauge}
Assume $\Omega$ is a Lipschitz polyhedral domain in $\mathbb{R}^{3}$. Then the following statements hold true:

(i) for any  $\frac{6}{5} \leq p < \frac{4}{3}$,  there exists a positive constant $C_{p}$ such that
\begin{align}
\label{estimate_2-1}
\Vert u\Vert_{W^{2,p}(\Omega)} \leq C_{p}  \Vert \Delta u\Vert_{L^{p}(\Omega)} ,
\qquad \forall u \in W^{2,p}(\Omega) \cap W_{0}^{1,p}(\Omega).
\end{align}
(ii)  if  $\Omega$ is also convex, then (\ref{estimate_2-1}) holds for any  $\frac{6}{5}< p\leq 2$.
\end{theorem}

\begin{remark}It is worth noticing that by Dahlberg's result ( see \cite{JK1995}), for any $p>1$, there exists some Lipschitz domain $\Omega$ and $f\in C^{\infty}(\overline{\Omega})$ such that the solution $u$ to
\begin{equation}
\label{eq2}
\left \{
 \begin{array}{ll}
\Delta u=f \ \ & \text{in}\ \Omega\\ \\
u=0 \  &\text{on}\ \partial\Omega
 \end{array}
\right.
 \end{equation} does not belong to $W^{2, p}(\Omega)$.  Hence the estimate (\ref{estimate_2-1}) usually does not hold for general Lipchitz domain.

\end{remark}

 Another important tool is the so-called Stein extension Theorem. Before stating the Stein extension Theorem, we would like to give some preliminary knowledge.
 \begin{lemma}(Theorem 2 of Chapter VI  in \cite{Stein1970})
 \label{lemma2}
 Let $F$ be any closed set in $\mathbb{R}^{n}$ with $n\geq2$ and $\delta(x)$ be the distance of $x$ from $F.$ Then there exists a function $\Delta(x)$ defined in $F^c$ such that \\

 $(a) \  c_1 \delta(x)\leq \Delta(x) \leq c_2 \delta(x), \ \ \forall \ x\in F^c $   \\

 $(b)$   $\Delta(x)$ is $C^{\infty}$ in $F^c$ and
 $$|\frac{\partial^{\alpha}}{\partial x^{\alpha}}\Delta(x)| \leq B_{\alpha} (\delta(x))^{1-|\alpha|}$$
 where $B_{\alpha}$, $c_1$ and $c_2$ are independent of $F$.  Here $\Delta(x)$ is called the regularized distance from $F$.
 \end{lemma}

 \begin{lemma}(Lemma 2 of Section 3 of  Chapter VI  in \cite{Stein1970})
 \label{lemma3}Suppose $D=\{(x, y)\in \mathbb{R}^{n+1}: y>\phi(x)\}$, where $\phi: \mathbb{R}^n \rightarrow \mathbb{R}$ is a Lipschitz function satisfying: $|\phi(x)-\phi(x')| \leq M|x-x'|$, for all $x, x'\in \mathbb{R}^n$. Then there exists a constant $c>0$,which depends only on the Lipschitz bound of $D$, so that if $(x, y)\in F^c$, then $c \Delta(x, y) \geq \phi(x)-y$.

 \end{lemma}\

 Let $$\delta^* =2c \Delta.$$
 Then by Lemma \ref{lemma3}, we see that $\delta^* (x, y)\geq \phi(x)-y$, if $(x, y)\in F^c$. Then we can state the Stein extension theorem as follows:

 \begin{theorem}( Theorem $5' $ in Chapter $VI$ in \cite{Stein1970})
\label{extension}
Let $D=\{(x, y)\in \mathbb{R}^{n+1}: y>\phi(x)\}$, where $\phi: \mathbb{R}^n \rightarrow \mathbb{R}$ is a Lipschitz function satisfying: $|\phi(x)-\phi(x')| \leq M|x-x'|$, for all $x, x'\in \mathbb{R}^n$ with $n\geq1$.
Then   there exists a bounded linear extension operator $\mathfrak{E}: W^{k, p}(D) \rightarrow W^{k, p}(\mathbb{R}^{n+1})$ for any  $1 \leq p \leq \infty$, and any nonnegative integer $k$. More precisely, for any $f\in W^{k, p}(D)\cap C^{\infty}(D)$, $\mathfrak{E}(f)(x, y)=f(x,y)$, if $(x, y)\in \overline{D}$ and  $$\mathfrak{E}(f)(x, y)= \int_1^{\infty}\psi(\lambda)f(x, y+\lambda\delta^*(x, y))d\lambda,\ \ \forall \ (x, y)\in {\overline{D}}^c.$$
Here $\psi$ is a continuous function on $[1, \infty)$ satisfying:$$\psi(\lambda)=O(\lambda^{-N}) \text{as $\lambda\rightarrow \infty$ for any $N$}, $$
$$\int_1^{\infty}\psi(\lambda) d\lambda=1\ \text{ and } \ \int_1^{\infty}\psi(\lambda)\lambda^k d\lambda=0\ \text{ for any $k=1,2, \cdots$.}$$
\end{theorem}

\vspace{30pt}

\section{proof of Main Results}
In this part, we will prove Theorem \ref{main_theorem1} and Corollary \ref{cor1}.

Firstly  we turn to the proof of Theorem \ref{main_theorem1}, which consists of several steps. In the beginning, we  choose $\Bx_{0} \in \partial \Omega$ arbitrarily.
Since $\Omega$ is a Lipschitz domain in $\mathbb{R}^{3}$,
there exists an open neighborhood $\CC_{R,H}$ of $\Bx_{0}$
in $\mathbb{R}^{3}$, and new orthogonal coordinates
$(x_{1},x_{2}, x_{3})$ such that in the new coordinates,
$\Bx_{0} = (0, 0, 0) \in \mathbb{R}^{3}$,  the neighborhood
can be represented by
\begin{align}
\label{Lipschitz_box}
\CC_{R,H} = \{ (x_{1}, x_{2}, 0)
+ t \Bv : (x_{1}, x_{2}) \in [-R, R]^{2},
-H < t < H\},
\end{align}
with the vector $\Bv = (0, 0, 1) \in \mathbb{R}^{3}$ , and
\begin{align*}
\CC_{R,H} \cap \Omega = & \CC_{R,H} \cap
\{ (x_{1}, x_{2},0) + t \Bv:
(x_{1},x_{2}) \in [-R,R]^{2},
t < \zeta (x_{1}, x_{2}) \}, \\
\CC_{R,H} \cap \partial\Omega = &\CC_{R,H} \cap
\{ (x_{1},x_{2},0) + t \Bv:
(x_{1}, x_{2}) \in [-R,R]^{2},
t = \zeta (x_{1}, x_{2}) \}, \\
\CC_{R,H} \cap \overline{\Omega}^{c} = & \CC_{R,H} \cap
\{ (x_{1}, x_{2},0) + t \Bv:
(x_{1}, x_{2}) \in [-R,R]^{2},
t > \zeta (x_{1}, x_{2}) \},
\end{align*}
for some Lipschitz function $\zeta : [-R,R]^{2}\rightarrow
\mathbb{R}$ satisfying
\begin{align}
\label{Lipschitzness1}
\zeta(0, 0) = 0, \text{ and } \vert \zeta (x_{1}, x_{2})
\vert \leq 0.05 H \quad \forall (x_{1}, x_{2})
\in [-R, R]^{2}.
\end{align}
Without losing of generality, we assume $0< R \leq H$. According to \cite[Theorem~$1$ of Section~$3.1.1$]{Evans1992}, $\zeta$ can be extended to be a global Lipschitz function on $\mathbb{R}^{2}$. We define
\begin{align}
\label{hypergraph_domain}
D = \{ \Bx \in \mathbb{R}^{3}: x_{3} < \zeta (x_{1}, x_{2})\}
\end{align}
and
\begin{align}
\label{Kone1}
\BK = \{ \Bx \in \mathbb{R}^{3} : x_{3} \geq
(L+1)\sqrt{x_{1}^{2} + x_{2}^{2}} \},
\end{align}
where $L$ is the Lipschitz constant of $\partial D$ introduced in (\ref{hypergraph_domain}).  We have

\begin{proposition}
\label{prop1}
For any $v \in W^{2,p}(\CC_{0.9R,0.9H}\cap \Omega)$, $v$ can be extended to $\CC_{0.9R, 0.9H}$ satisfying $v \in W^{2,p}(\CC_{0.9R, 0.9H})$, and there exists a uniform positive constant $C_{1}$ such that for any $k=0,1,2$ and any $\sigma \in [0.2, 0.9]$,  there holds:
\begin{align}
\label{claim1}
\Vert v\Vert_{W^{k,p}(\CC_{\sigma R,\sigma H} \cap
\overline{\Omega}^{c} )} \leq
C_{1}\Vert v\Vert_{W^{k,p}(\CC_{\sigma R,0.2 H} \cap \Omega)}.
\end{align}
 \end{proposition}

\begin{proof}
Obviously, $\CC_{0.9R,0.9H}\cap \Omega$ is a Lipschitz domain in $\mathbb{R}^{3}$. Thus by Theorem \ref{extension}, $v$ can be extended to $D$ (see (\ref{hypergraph_domain})) such that $v \in W^{2,p}(D)$
and for any $k=0,1,2$
\begin{align*}
\Vert v\Vert_{W^{k,p}(D)} \leq \tilde{C}_{1}
\Vert v\Vert_{W^{k,p}(\CC_{0.9R,0.9H} \cap \Omega)}.
\end{align*}

Let $\overline{\eta}: \mathbb{R}\rightarrow\mathbb{R}$ be a  smooth function such that $\overline{\eta}(x_{3}) = 1$ for any $x_{3} \geq -0.1H$, $0 \leq \overline{\eta}(x_{3}) \leq 1$ for any $x_{3} \in \mathbb{R}$, and $\overline{\eta}(x_{3}) = 0$ for any $x_{3} \leq - 0.2 H$.
We define
\begin{align*}
w(\Bx) = \overline{\eta}(x_{3})v(\Bx),\qquad \forall
\Bx=(x_1, x_2, x_3) \in D.
\end{align*}
By (\ref{Lipschitzness1}) and the definition of ${\eta}$, it is easy to see that for any $(x_{1},x_{2})\in [-R, R]^{2}$,
\begin{align}
\label{w_v_eq1}
-0.1H < \zeta (x_{1}, x_{2}) \text{ and }
 w(\Bx) = v(\Bx), \quad \forall
 -0.1 H\leq x_{3} \leq \zeta (x_{1}, x_{2}).
\end{align}
In addition, for any $(x_{1},x_{2})\in [-R, R]^{2}$, there holds:
\begin{align*}
w(\Bx) = 0,  \quad \forall -H \leq x_{3} \leq -0.2H.
\end{align*}

For simplicity, we firstly assume
$w |_{\CC_{0.9 R, 0.9H}\cap \Omega}\in C^{\infty}
(\overline{\CC_{0.9 R, 0.9 H}\cap \Omega})$.
We extend $w$ from
$\CC_{0.9R,0.9H}\cap \Omega$ to $\CC_{0.9R,0.9H}$ by the
following extension
\begin{align}
\label{Stein_extension}
\mathfrak{E}w(\Bx) = \int_{1}^{\infty} w(x_{1}, x_{2}, x_{3}
+ \lambda \delta^{*}(\Bx))\psi (\lambda) d\lambda, \qquad
\forall \Bx \in \CC_{0.9R,0.9H} \cap \overline{\Omega}^{c}.
\end{align}
which is exactly the Stein extension in Theorem \ref{extension}. By (\ref{Stein_extension}), it is easy to see that for any $\Bx \in \CC_{0.9R,0.9H} \cap \overline{\Omega}^{c}$,
$\mathfrak{E}w(\Bx)$ relies only on the value of
$w(x_{1}, x_{2}, z)= \overline{\eta}(z) v(x_{1}, x_{2}, z)$
where $z \in (-0.2H, \zeta(x_{1}, x_{2})]$.  Therefore by
Theorem \ref{extension}, $w\in W^{2,p}(\CC_{0.9R,0.9H})$ and there exists a positive constant
$C_{1}^{\prime}$ such that for any $k=0,1,2$,
\begin{align}
\label{local_extension_ineq1}
\Vert w\Vert_{W^{k,p}(\CC_{\sigma R, \sigma H})} \leq &
C_{1}^{\prime} \Vert w\Vert_{W^{k,p}(\CC_{\sigma R, 0.2 H}
\cap \Omega)},
\end{align}
for any $0.2 \leq \sigma \leq 0.9$. Here $C_{1}^{\prime}$ is
independent of $\sigma \in [0.2 , 0.9]$.

In the following  the restriction that $w \in C^{\infty}
(\overline{\CC_{R,H}\cap \Omega})$ would be removed by approximation of mollifying (see also Section 3.2.4 in Chapter VI of \cite{Stein1970}). Let $\tilde{\eta}\in C_0^{\infty}(\mathbb{R}^{3})$ satisfy: $0\leq \tilde{\eta}\leq 1$ on $\mathbb{R}^{3}$,
$\int_{\mathbb{R}^{3}} \tilde{\eta}(x) \mathrm{d}x = 1$ and $\mathrm{supp}\tilde{\eta}\subset \BK^{o}$, where $\BK^o$ denotes the interior of the cone $\BK$ introduced in (\ref{Kone1}).
 $\forall \ \epsilon > 0$, we define $\tilde{\eta}(_{\epsilon}\Bx)
= \epsilon^{-3}\tilde{\eta}(\frac{\Bx}{\epsilon})$
for any $\Bx \in \mathbb{R}^{3}$ and
\begin{align*}
w_{\epsilon} (\Bx) = \int_{\text{supp} \tilde{\eta}_{\epsilon}}
w(\Bx - \By) \tilde{\eta}_{\epsilon}(\By) d\By, \qquad
\forall \Bx \in \CC_{0.9R,0.9H} \cap \Omega.
\end{align*}
It is easy to see that there exists $\epsilon_{0}>0$
such that for any $0< \epsilon < \epsilon_{0}$, we have
\begin{align*}
& \Bx - \By \in \CC_{R, H} \cap \Omega, \qquad \forall
\Bx \in \overline{\CC_{0.9R, 0.9H}\cap \Omega},
  \ \By \in \text{supp} \tilde{\eta}_{\epsilon}, \\
& x_{3} - y_{3} < -0.2 H,\qquad \forall
\Bx \in \overline{\CC_{0.9R, 0.9H}\cap \Omega}
\text{ with } x_{3} < -0.2H,  \ \By \in \text{supp} \tilde{\eta}_{\epsilon}.
\end{align*}
Thus for any $0 < \epsilon < \epsilon_{0}$,
$w_{\epsilon}$ is well-defined in $\CC_{0.9R, 0.9H}\cap \Omega$,
$\text{supp}\left(w_{\epsilon}|_{\CC_{0.9R, 0.9H}\cap \Omega}
\right) \subset \overline{\CC_{0.9R, 0.2H}\cap \Omega}$ and $w_{\epsilon}$
is a smooth function in an open neighborhood of
$\overline{\CC_{0.9R, 0.9H}\cap \Omega}$. Therefore,   $\mathfrak{E}w_{\epsilon}$  by (\ref{Stein_extension}) is a well-defined extension of $w_{\epsilon}$ from $\CC_{0.9R, 0.9H}\cap \Omega$
to $\CC_{0.9R, 0.9H}$.
In addition, we have that for any $k=0,1,2$,
\begin{align*}
\Vert w_{\epsilon} - w\Vert_{W^{k,p}(\CC_{0.9R, 0.9H}\cap \Omega)}
\rightarrow 0 \text{ as } \epsilon \rightarrow 0
\end{align*}
which shows , for any $k=0,1,2$,
\begin{align}
\label{uniform_smooth_conv}
\Vert w_{\epsilon} - w\Vert_{W^{k,p}(\CC_{\sigma R, \sigma H}\cap
\Omega)} \rightarrow 0 \text{ uniformly for any } \sigma \in [0.2, 0.9],
\end{align}
as $\epsilon \rightarrow 0$.
On the other hand, by (\ref{local_extension_ineq1}) and the fact that
$\text{supp}\left(w_{\epsilon}|_{\CC_{0.9R, 0.9H}\cap \Omega} \right)
\subset \overline{\CC_{0.9R, 0.2H}\cap
\Omega}$, we have that
for any $0 < \epsilon < \epsilon_{0}$,
\begin{align}
\label{uniform_smooth_bound}
\Vert w_{\epsilon}\Vert_{W^{k,p}(\CC_{\sigma R, \sigma H})} \leq &
C_{1}^{\prime} \Vert w_{\epsilon}\Vert_{W^{k,p}(\CC_{\sigma R, 0.2 H}
\cap \Omega)}
\end{align}
for any $\sigma \in [0.2, 0.9]$.
We define $\mathfrak{E}w (\Bx) := \lim_{\epsilon \rightarrow 0}
w_{\epsilon}(\Bx)$ for any $\Bx \in \CC_{0.9R, 0.9H}\cap
\overline{\Omega}^{c}$.
By (\ref{uniform_smooth_conv}) and (\ref{uniform_smooth_bound}), it infers that $w \in W^{2,p}(\CC_{0.9R, 0.9H})$ and
for any $k=0,1,2$,
\begin{align}
\label{local_extension_ineq2}
\Vert w\Vert_{W^{k,p}(\CC_{\sigma R, \sigma H})}
\leq & C_{1}^{\prime\prime}
\Vert w \Vert_{W^{k,p}(\CC_{\sigma R, 0.2 H}\cap \Omega)}
\end{align}
for any $\sigma \in [0.2, 0.9]$. Here the constant $C_{1}^{\prime\prime}$
is independent of $\sigma \in [0.2 , 0.9]$.

Now, we extend any $v\in W^{2,p}(\CC_{R,H}\cap \Omega)$ from
$\CC_{R,H}\cap \Omega$ to $\CC_{0.9R,0.9H}$ by
setting $v(\Bx) = w(\Bx)$ for any $\Bx \in \CC_{0.9R,0.9H} \cap
\overline{\Omega}^{c}$. According to (\ref{w_v_eq1}) and
(\ref{local_extension_ineq2}), we have that
$v \in W^{2,p}(\CC_{0.9R,0.9H})$ and there exists a positive constant $C_{1}$ such that
\begin{align*}
\Vert v\Vert_{W^{2,p}(\CC_{\sigma R,\sigma H}\cap\overline{\Omega}^{c})}
\leq C_{1} \Vert v\Vert_{W^{2,p}(\CC_{\sigma R, 0.2 H}\cap \Omega)},
\end{align*}
for any $0.2 \leq \sigma \leq 0.9$, where the constant $C_{1}$ is
independent of $\sigma \in [0.2 , 0.9]$. Theorefore, the claim
(\ref{claim1}) holds.
\end{proof}\vspace{10pt}

The following proposition can be viewed a extension of Theorem \ref{dauge}.
\begin{proposition}
\label{prop2}Let $\Omega\subseteq\mathbb{R}^3$ be a Lipschitz polyhedra. Then

$(1)$ for any $6/5\leq p<4/3$, there exits a positive constant $\delta>0$ such that if  $2R^{2}+H^{2} \leq \delta^{2} \leq 1$, there holds:
\begin{align}
\label{Dauge_ineq2}
\Vert v\Vert_{W^{2,p}(\Omega)} \leq 2 \tilde{C}_{p}
\Vert A:D^{2}v\Vert_{L^{p}(\Omega)}, \quad
\end{align}for any $v \in W^{2,p}(\Omega)\cap W_{0}^{1,p}(\Omega)$ with
$\mathrm{supp}v \subset \overline{\CC_{R, H}\cap \Omega}$ .

$(2)$ if $\Omega$ is also convex, then the conclusion of $(1)$ still holds for any $6/5\leq p\leq 2$.
\end{proposition}

\begin{proof}We denote by $L_{0}$ the constant coefficient operator given by
\begin{align*}
L_{0} u = A(\Bx_{0})D^{2}u.
\end{align*}
Since $A(\Bx_{0})$ is symmetric and positive definite, there is an
invertible matrix $B \in \mathbb{R}^{3 \times 3}$ such that
\begin{align*}
B^{\top} A(\Bx_{0}) B = I_{3}.
\end{align*}
We define the affine mapping $G:\mathbb{R}^{3}\rightarrow
\mathbb{R}^{3}$ as
\begin{align*}
\Bx = G(\widehat{\Bx}) := B^{-1}\widehat{\Bx} + \Bx_{0}, \qquad
\forall \widehat{\Bx} \in \mathbb{R}^{3}.
\end{align*}
We denote by $\widehat{\Omega}$ the preimage of $\Omega$ via the above
affine mapping $G$. For any $v \in W^{2,p}(\Omega) \cap
W_{0}^{1,p}(\Omega)$, we define $\widehat{v}(\widehat{\Bx}) :=
v (G(\widehat{\Bx}))$ for any $\widehat{\Bx} \in \widehat{\Omega}$.
It's easy to see that
\begin{align*}
A(\Bx_{0}): D^{2}u (G(\Bx)) = I_{d}: D_{\widehat{\Bx}}^{2}
\widehat{u}(\widehat{\Bx}), \qquad \forall \widehat{\Bx}
\in \widehat{\Omega}.
\end{align*}
According to Theorem \ref{dauge}, for any $p$ satisfying: either (1) $\frac{6}{5}
\leq p < \frac{4}{3}$, or (2) $\frac{6}{5} \leq p \leq 2$ when $\Omega$ is also convex,
there is a positive constant $\widehat{C}_{p}$ depending only on $\Bx_0$, $\Omega$
and $p$ such that
\begin{align*}
\Vert D_{\widehat{\Bx}}^{2} \widehat{v}\Vert_{L^{p}(\widehat{\Omega})}
+\Vert D_{\widehat{\Bx}} \widehat{v}\Vert_{L^{p}(\widehat{\Omega})}
+ \Vert \widehat{v}\Vert_{L^{p}(\widehat{\Omega})}
\leq \widehat{C}_{p} \Vert I_{3}: D_{\widehat{\Bx}}^{2}
\widehat{v}\Vert_{L^{p}(\widehat{\Omega})}.
\end{align*}
Therefore, there is $\tilde{C}_{p} > 0$ such that
\begin{align}
\label{Dauge_ineq1}
\Vert v\Vert_{W^{2,p}(\Omega)} \leq \tilde{C}_{p}
\Vert L_{0} v\Vert_{L^{p}(\Omega)}, \qquad
\forall v \in W^{2,p}(\Omega) \cap W_{0}^{1,p}(\Omega)
\end{align}
where the constant $\tilde{C}_{p}$ depends only on $\Bx_{0}$, $\Omega$ and $p$.
Consequently, if $\text{supp}v \subset \overline{\CC_{R,H}\cap\Omega}$,
there holds
\begin{align*}
L_{0} v = (A(\Bx_{0}) - A): D^{2}v + A:D^{2}v,
\end{align*}
and by (\ref{Dauge_ineq1}), it infers
\begin{align*}
\Vert v\Vert_{W^{2,p}(\Omega)}
\leq \tilde{C}_{p} \left( \sup_{\CC_{R,H}\cap \Omega}
\vert A - A(\Bx_{0})\vert \Vert D^{2}v \Vert_{L^{p}(\Omega)}
+ \Vert A:D^{2}v\Vert_{L^{p}(\Omega)} \right).
\end{align*}
Since $A \in [C^{0}(\overline{\Omega})]^{3 \times 3}$, there exists
a positive number $0< \delta \leq 1$ such that
\begin{align*}
\vert A(\Bx) - A(\Bx_{0}) \vert \leq \dfrac{1}{2(\tilde{C}_{p}+1)},
\qquad \forall \Bx \in B_{\delta}(\Bx_{0}) \cap \Omega,
\end{align*}
and hence
\begin{align*}
\Vert v\Vert_{W^{2,p}(\Omega)} \leq 2 \tilde{C}_{p}
\Vert A:D^{2}v\Vert_{L^{p}(\Omega)}, \quad \forall
v \in W^{2,p}(\Omega)\cap W_{0}^{1,p}(\Omega),
\text{supp}v \subset \overline{\CC_{R, H}\cap \Omega},
\end{align*}
provided $2R^{2}+H^{2} \leq \delta^{2} \leq 1$.
\end{proof}\

Next, we finish the proof of Theorem \ref{main_theorem1}.

{\bf Proof of Theorem \ref{main_theorem1}:} For simplicity,  for any $t>0$, we denote $\CC_{tR, tH}$ by $\CC_{t}$.
 For any $ \sigma \in (0.2, 0.9)$, we  introduce a cutoff function
$\eta \in C_{0}^{2}(\CC_{1})$ satisfying
\begin{align*}
0 \leq \eta \leq 1 \ \text{on}\ \mathbb{R}^3,   \  \eta \equiv 1 \ \text{ in} \ \CC_{\sigma}, \ \eta = 0 \ \text{ in }\ \CC_{1} \backslash \CC_{\frac{0.9+\sigma}{2}} ,\\
 \ \  {4}{(0.9-\sigma)R}\vert D \eta\vert   +   (0.9-\sigma)^{2}R^{2}\vert D^{2} \eta\vert \leq 16.
\end{align*}Here we have used the assumption: $R\leq H.$

By Proposition \ref{prop2}  and taking $v = \eta u$ in (\ref{Dauge_ineq2}), there exists a small number $\delta>0$ such that when $2R^{2} + H^{2} \leq \delta^{2} \leq 1$, there holds:
\begin{align*}
& \Vert u\Vert_{W^{2,p}(\CC_{\sigma}\cap \Omega)}
\leq \Vert \eta u\Vert_{W^{2,p}(\Omega)} \\
\leq & 2\tilde{C}_{p} \Vert A:D^{2}(\eta u)\Vert_{L^{p}(\Omega)}\\
= & 2\tilde{C}_{p} \Vert \eta A:D^{2}u + 2 a^{ij}D_{i}\eta
D_{j}u + u A:D^{2}\eta \Vert_{L^{p}(\CC_{0.9}\cap \Omega)} \\
\leq & 32 \tilde{C}_{p} \big( \Vert A:D^{2}u\Vert_{L^{p}
(\CC_{0.9}\cap \Omega)} + {((0.9-\sigma)R)^{-1}}
\Vert D u\Vert_{L^{p}((\CC_{\frac{0.9+\sigma}{2}}
\cap\Omega) \backslash \CC_{\sigma})} \\
& \qquad \qquad
+ {(0.9-\sigma)^{-2}R^{-2}} \Vert u\Vert_{L^{p}(\CC_{0.9}
\cap\Omega)}\big)
\end{align*}  where $\tilde{C}_{p}$ is independent of $\sigma, R$ and $H$.
Therefore if $2R^{2} + H^{2} \leq \delta^{2} \leq 1$,
we have
\begin{align*}
& (0.9 - \sigma)^{2}R^{2}\Vert u\Vert_{W^{2,p}(\CC_{\sigma }
\cap \Omega)} \\
\leq & 32 \tilde{C}_{p}\big( R^{2}
\Vert A:D^{2}u\Vert_{L^{p}
(\CC_{0.9}\cap \Omega)} + (0.9 - \sigma)R
\Vert D u\Vert_{L^{p}((\CC_{\frac{0.9+\sigma}{2}}
\cap\Omega)
\backslash \CC_{\sigma})}
+ \Vert u\Vert_{L^{p}(\CC_{0.9}\cap\Omega)}\big).
\end{align*}
Since $u|_{\CC_{0.9}\cap \Omega} \in W^{2,p}
(\CC_{0.9}\cap \Omega)$, we extend $u|_{{\CC_{0.9}}\cap\Omega}$
from  ${{\CC_{0.9}}\cap\Omega}$ to $\CC_{0.9}$ by the extension satisfying (\ref{claim1}).
Thus there exists a constant $C_{0}$ such that
\begin{align*}
& (0.9 - \sigma)^{2} R^{2}\Vert u\Vert_{W^{2,p}(\CC_{\sigma })}
\leq C_{0} (0.9 - \sigma)^{2} R^{2} \Vert u\Vert_{W^{2,p}
(\CC_{\sigma}\cap \Omega)} \\
\leq & 32 C_{0} \tilde{C}_{p}\big(
R^{2} \Vert A:D^{2}u\Vert_{L^{p}
(\CC_{0.9}\cap \Omega)} + (0.9 - \sigma)R
\Vert D u\Vert_{L^{p}((\CC_{\frac{0.9+\sigma}{2}}
\cap\Omega) \backslash \CC_{\sigma})}
+ \Vert u\Vert_{L^{p}(\CC_{0.9}\cap\Omega)}\big) \\
\leq & 32 C_{0} \tilde{C}_{p}\big(
R^{2} \Vert A:D^{2}u\Vert_{L^{p}
(\CC_{0.9}\cap \Omega)} + (0.9 - \sigma)R
\Vert D u\Vert_{L^{p}(\CC_{\frac{0.9+\sigma}{2}}
\backslash \CC_{\sigma})}
+ \Vert u\Vert_{L^{p}(\CC_{0.9})}\big).
\end{align*}

By \cite[Theorem~$7.28$]{GT01} with standard scaling argument, for any $\epsilon>0$, there holds:
\begin{align*}
& (0.9 - \sigma)R
\Vert D u\Vert_{L^{p}(\CC_{\frac{1+\sigma}{2}}
\setminus \CC_{\sigma })} \\
\leq & \epsilon (0.9 - \sigma)^{2}R^{2}
\Vert D^{2} u\Vert_{L^{p}(\CC_{\frac{1+\sigma}{2}}
\backslash \CC_{\sigma})}
+ {C_{2}}{\epsilon}^{-1} \Vert u\Vert_{L^{p}
(\CC_{\frac{1+\sigma}{2}}
\backslash \CC_{\sigma})},
\end{align*}where $C_{2}$ is independent of $\epsilon$ and $\sigma$.
 Therefore, for any $\epsilon > 0$, we have
\begin{align*}
& (0.9 - \sigma)^{2}R^{2}\Vert u\Vert_{W^{2,p}
(\CC_{\sigma})}\\
\leq & 32 C_{0} \tilde{C}_{p} \big( R^{2}
\Vert A:D^{2}u\Vert_{L^{p}(\CC_{0.9}\cap \Omega)}
+ \epsilon (0.9 - \sigma)^{2}R^{2}
\Vert D^{2} u\Vert_{L^{p}(\CC_{\frac{0.9+\sigma}{2}})} \\
& \qquad \qquad \qquad + (C_{2}+1){\epsilon}^{-1}
\Vert u\Vert_{L^{p}(\CC_{0.9})}\big).
\end{align*}

Introducing the weight semi-norms
\begin{align*}
\Phi_{k} = \sup_{0.2 < \sigma < 0.9} (0.9-\sigma)^{k}R^{k}
\Vert D^{k}u\Vert_{L^{p}(\CC_{\sigma})},\quad
k = 0, 1, 2.
\end{align*}
Therefore we obtain that for any $0.2 < \sigma < 0.9$ and
any $\epsilon >0$,
\begin{align*}
(0.9 - \sigma)^{2}R^{2}\Vert u\Vert_{W^{2,p}(\CC_{\sigma })}
\leq 32 C_{0}\tilde{C}_{p} \left( R^{2}\Vert A:D^{2}u\Vert_{L^{p}
(\CC_{0.9}\cap \Omega)} + \epsilon\Phi_{2} +
\frac{C_{2}+1}{\epsilon}\Phi_{0}\right).
\end{align*}
Since the above inequality holds for any $0.2 < \sigma < 0.9$, we have
that for any $\epsilon > 0$,
\begin{align*}
\Phi_{2} \leq 32 C_{0} \tilde{C}_{p} \left(
R^{2}\Vert A:D^{2}u\Vert_{L^{p}
(\CC_{0.9}\cap \Omega)} + \epsilon\Phi_{2}
+ \frac{C_{2} + 1}{\epsilon} \Phi_{0} \right)
\end{align*}
provided $2R^{2} + H^{2} \leq \delta^{2} \leq 1$.
By taking $\epsilon>0$ small enough
in the latest inequality above, we obtain
\begin{align}
\label{energy_ineq1}
\Phi_{2} \leq 64 C_{0}\tilde{C}_{p}\left(
R^{2}\Vert A:D^{2}u\Vert_{L^{p}
(\CC_{0.9}\cap \Omega)}
+ C_{3}^{\prime}\Vert u\Vert_{L^{p}(\CC_{0.9})}\right)
\end{align}
provided $2R^{2}+H^{2} \leq \delta^{2} \leq 1$.

 Therefore by (\ref{energy_ineq1}),  if $2R^{2}+H^{2} \leq \delta^{2}$,  there holds:
\begin{align}
\label{energy_ineq2}
\nonumber & (0.9 - \sigma)^{2}R^{2}\Vert D^{2}u\Vert_{L^{p}
(\CC_{\sigma} \cap \Omega)} \\
\nonumber
 \leq & 32 C_{0} \tilde{C}_{p}
 \left(  R^{2}\Vert A: D^{2}u\Vert_{L^{p}
(\CC_{0.9}\cap \Omega)}
+ C_{3}^{\prime}\Vert u\Vert_{L^{p}(\CC_{0.9})} \right) \\
\leq& C \left(  R^{2}\Vert A:D^{2}u\Vert_{L^{p}
(\Omega)}
+ C_{3}\Vert u\Vert_{L^{p}(\Omega)} \right)
\end{align}
for any $0 < \sigma < 0.9$.

The proof can be completed by taking $\sigma = \frac{1}{2}$ in
(\ref{energy_ineq2}) and covering $\partial\Omega$ with a
finite number of such neighbourhoods.  \hfill $\square$\par
\vspace{10pt}

Finally, we would like to conclude the proof of Corollary \ref{cor1}, which is similar that of Theorem 9.17 in Gilbarg-Trudinger \cite{GT01}. For completeness, we outline the main idea of proof here.
\vspace{5pt}

\textbf{Proof of Corollary \ref{cor1}}: We argue by contradiction. For simplicity, we write $Lu=A:D^2u$. If the conclusion does not hold true, then there must exist a sequence $\{v_m\}\subset W^{2, p}(\Omega)\cap W^{1,p}_0(\Omega) $ such that $\|v_m\|_{L^p (\Omega)}=1$ for any $m$ and $\|Lv_m\|_{L^p(\Omega)} \rightarrow 0$ as $m\rightarrow \infty$. From Theorem \ref{main_theorem1},$\{\|v_m\|_{W^{2,p }(\Omega)} \}$ is uniformly bounded. Then  by compact embedding of $W_0^{1,p}\rightarrow L^p$, there is a subsequence, still denoted by $\{v_m\}$, converging weakly to $v\in W^{2, p}(\Omega)\cap W^{1,p}_0(\Omega)$ with $$\|v\|_{L^p (\Omega)}=1\ \ \ \  \text{and}\  Lv=0 .$$
Since $p>3/2$, we see that $v\in C^0(\overline{\Omega})$ by Sobolev embedding theorem.  Then by interior $W^{2, p}$ estimate (Theorem 9.11 in \cite{GT01}), we  obtain $v\in W_{\mathrm{loc}}^{2, 3}$. Hence by Alexandrov maximum principle and the fact: $ v=0 \ \text{on}\ \partial\Omega$, $v=0$ in $\Omega$, which contradicts with the condition: $\|v\|_{L^p (\Omega)}=1$.

\bibliographystyle{plain}

\end{document}